\documentclass{amsart}

\usepackage{begnac}

\begin{document}

\title{On uniform canonical bases in $L_p$ lattices and other metric structures}

\author{Itaï \textsc{Ben Yaacov}}

\address{Itaï \textsc{Ben Yaacov} \\
  Université Claude Bernard -- Lyon 1 \\
  Institut Camille Jordan, CNRS UMR 5208 \\
  43 boulevard du 11 novembre 1918 \\
  69622 Villeurbanne Cedex \\
  France}

\urladdr{\url{http://math.univ-lyon1.fr/~begnac/}}

\thanks{Author supported by
  ANR chaire d'excellence junior THEMODMET (ANR-06-CEXC-007) and
  by the Institut Universitaire de France.}

\svnInfo $Id: UniformCB.tex 1354 2012-05-21 08:38:39Z begnac $
\thanks{\textit{Revision} {\svnInfoRevision} \textit{of} \today}

\keywords{stable theory, uniform canonical base, $L_p$ Banach lattice, beautiful pairs, metric valued fields}
\subjclass[2000]{03C45, 46B42, 12J25}

\begin{abstract}
  We discuss the notion of \emph{uniform canonical bases}, both in an abstract manner and specifically for the theory of atomless $L_p$ lattices.
  We also discuss the connection between the definability of the set of uniform canonical bases and the existence of the theory of beautiful pairs (i.e., with the finite cover property), and prove in particular that the set of uniform canonical bases is definable in algebraically closed metric valued fields.
\end{abstract}

\maketitle

\section*{Introduction}

In stability theory, the \emph{canonical base} of a type is a minimal set of parameters required to define the type, and as such it generalises notions such as the field of definition of a variety in algebraic geometry.
Just like the field of definition, the canonical base is usually considered as a set, a point of view which renders it a relatively ``coarse'' invariant of the type.
We may ask, for example, whether a type is definable over a given set (i.e., whether the set contains the canonical base), or whether the canonical base, as a set, is equal to some other set.
However, canonical bases, viewed as sets, cannot by any means classify types over a given model of the theory, and they may very well be equal for two distinct types.
The finer notion of \emph{uniform} canonical bases, namely, of canonical bases from which the types can be recovered uniformly, is a fairly natural one, and has appeared implicitly in the literature in several contexts (e.g., from the author's point of view, in a joint work with Berenstein and Henson \cite{BenYaacov-Berenstein-Henson:AlmostIndiscernible}, where convergence of uniform canonical bases is discussed).

Definitions regarding uniform canonical bases and a few relatively easy properties are given in \fref{sec:UniformCB}.
In particular we observe that every stable theory admits uniform canonical bases \emph{in some imaginary sorts}, so the space of all types can be naturally identified with a type-definable set.
We then turn to discuss the following two questions.

The first question is whether, for one concrete theory or another, there exist \emph{mathematically natural} uniform canonical bases, namely, uniform canonical bases consisting of objects with a clear mathematical meaning.
A positive answer may convey additional insight into the structure of the space of types as a type-definable set.
This is in contrast with the canonical parameters for the definitions, whose meaning is essentially tautological and can therefore convey no further insight.
The case of Hilbert spaces is quite easy, and merely serves as a particularly accessible example.
The case of atomless probability spaces (i.e., probability algebras, or spaces of random variables), treated in \fref{sec:RV}, is not much more difficult.
Most of the work is spent in \fref{sec:Lp} where we construct uniform canonical bases for atomless $L_p$ lattices in the form of ``partial conditional expectations'' $\bE_t[\cdot|E]$ and $\bE_{[s,t]}[\cdot|E]$ (defined there).
To a large extent, it is this last observation which prompted the writing of the present paper.

The second question, discussed in \fref{sec:BeautifulPairs}, is whether the (type-definable) set of uniform canonical bases is in fact definable.
We characterise this situation in terms of the existence of a theory of beautiful pairs.
In \fref{sec:PairsInACMVF} we use earlier results to show that for the theory of algebraically closed metric valued fields, the theory of beautiful pairs does indeed exist, and therefore that the sets of uniform canonical bases (which we do not describe explicitly) are definable.

\medskip

For stability in the context of classical logic we refer the reader to Pillay \cite{Pillay:GeometricStability}.
Stability in the context of continuous logic, as well as the logic itself, are introduced in \cite{BenYaacov-Usvyatsov:CFO}.

\section{Uniform canonical bases}
\label{sec:UniformCB}

In classical logic, stable theories are characterised by the property that for every model $\cM$, every type $p(\bar x) \in \tS_{\bar x}(M)$ is definable, i.e., that for each formula $\varphi(\bar x,\bar y)$ (say without parameters, this does not really matter) there exists a formula $\psi(\bar y)$ (with parameters in $M$) such that for all $\bar b \in M$:
\begin{gather*}
  \varphi(\bar x,\bar b) \in p \qquad \Longleftrightarrow \qquad \vDash \psi(\bar b).
\end{gather*}
In this case we say that $\psi$ is the $\varphi$-definition of $p$, and write
\begin{gather*}
  \psi(\bar y) \quad = \quad d_{p(\bar x)}\varphi(\bar x,\bar y).
\end{gather*}
Obviously, there may exist more than one way of writing a $\varphi$-definition for $p$, but since any two such definitions are over $\cM$ and equivalent there, they are also equivalent in every elementary extension of $\cM$, and thus have inter-definable canonical parameters.
In other words, the canonical parameter of the $\varphi$-definition of $p$ is well-defined, up to inter-definability, denoted $\Cb_\varphi(p)$.
The collection of all such canonical parameters, as $\varphi(\bar x,\bar y)$ varies (and so does $\bar y$) is called the \emph{canonical base} of $p$, denoted $\Cb(p)$.
This is, up to inter-definability, the (unique) smallest set over which $p$ is definable.
The same holds for continuous logic with some minor necessary changes, namely that the $\varphi$-definition may be a definable predicate (i.e. a uniform limit of formulae, rather than a formula), and it defines $p$ in the sense that
\begin{gather*}
  \varphi(\bar x,\bar b)^p \quad = \quad d_{p(\bar x)}\varphi(\bar x,\bar b).
\end{gather*}
We shall hereafter refer to definable predicates as formulae as well, since for our present purposes the distinction serves no useful end.

Since canonical parameters are, \textit{a priori}, imaginary elements, the canonical base is a subset of $M^{eq}$.
For most purposes of abstract model theory this is of no hindrance, but when dealing with a specific theory with a natural ``home sort'', it is interesting (and common) to ask whether types admit canonical bases which are subsets of the model.
This is true, of course, in any stable theory which eliminates imaginaries.
In continuous logic, this is trivially true for Hilbert spaces, it is proved for probability algebras in \cite{BenYaacov:SchroedingersCat}, and for $L_p$ Banach lattices in \cite{BenYaacov-Berenstein-Henson:LpBanachLattices} (so all of these theories have, in particular, weak elimination of imaginaries, even though not full elimination of imaginaries).

A somewhat less commonly asked question is the following.
Can we find, for each formula $\varphi(\bar x,\bar y)$, a formula $d\varphi(\bar y,Z)$, where $Z$ is some infinite tuple of variables of which only finitely (or countably) many actually appear in $d\varphi$, such that for every model $\cM$, and every type $p(\bar x) \in \tS_{\bar x}(M)$,
\begin{gather*}
  d_{p(\bar x)}\varphi(\bar x,\bar y) \quad = \quad d\varphi\bigl( \bar y, \Cb(p) \bigr).
\end{gather*}
The scarcity of references to this question is actually hardly surprising, since, first, the question as stated makes no sense, and, second, the answer is positive for every stable theory.
Indeed, if we consider $\Cb(p)$ to be merely a set which is only known up to inter-definability, as is the common practice, then the expression $d\varphi\bigl( \bar y, \Cb(p) \bigr)$ is meaningless.
We remedy this in the following manner:

\begin{dfn}
  \label{dfn:UniformCB}
  Let $T$ be a stable theory.
  A \emph{uniform definition of types} in the sort of $\bar x$ consists of a family of formulae $\bigl\{ d\varphi(\bar y,Z) \bigr\}
  _{\varphi(\bar x,\bar y) \in \cL}$, where $Z$ is a possibly infinite tuple, such that for each type $p(\bar x)$ over a model $\cM \vDash T$ there exists a tuple $A \subseteq M^{eq}$ in the sort of $Z$ such that for each $\varphi(\bar x,\bar y)$:
    \begin{gather*}
      d_{p(\bar x)}\varphi(\bar x,\bar y)
      \quad = \quad
      d\varphi( \bar y, A).
    \end{gather*}
  If, in addition, this determines the tuple $A$ uniquely for each $p$ then we write $A = \Cb(p)$ and say that the map $p \mapsto \Cb(p)$ is a \emph{uniform canonical base map} (in the sort of $\bar x$), or that the canonical bases $\Cb(p)$ are \emph{uniform} (in $p(\bar x)$).

  To complement the definition, a (non uniform) \emph{canonical base map} is any map $\Cb$ which associates to a type $p$ over a model some tuple $\Cb(p)$ which enumerates a canonical base for $p$.
\end{dfn}

First of all, we observe that every uniform canonical base map is in particular a canonical base map.
Second, any uniform definition of types gives rise naturally to a uniform canonical base map.
Indeed, for each $\varphi$ we let $w_\varphi$ be a variable in the sort of canonical parameters for $d\varphi(\bar y,Z)$, and let $d\varphi'(\bar y,w_\varphi)$ be the corresponding formula.
For a type $p$, let $A$ be a parameter for the original definition, and for each $\varphi$ let $b_\varphi$ be the canonical parameter of $d\varphi(\bar y,A)$, so $d\varphi(\bar y,A) = d\varphi'(\bar y,b_\varphi)$.
Now let $W$ be the tuple consisting of all such $w_\varphi$, so we may re-write $d\varphi'(\bar y,w_\varphi)$ as $d\varphi'(\bar y,W)$, and let $B$ be the tuple consisting of all such $b_\varphi$.
Then $d_{p(\bar x)}\varphi(\bar x,\bar y)
= d\varphi(\bar y,A) = d\varphi'(\bar y,B)$ for all $\varphi$, and in addition this determines $B$ uniquely.
Thus $\Cb(p) = B$ is a uniform canonical base map.

\begin{lem}
  \label{lem:ExistUniformCB}
  Every stable theory admits uniform definitions of types and thus uniform canonical base maps (in every sort).
\end{lem}
\begin{proof}
  This is shown for classical logic in, say, \cite{Pillay:GeometricStability}, and for continuous logic (which encompasses classical logic as a special case) in \cite{BenYaacov-Usvyatsov:CFO}.
\end{proof}

\begin{lem}
  The image $\img \Cb$ of a uniform canonical base map is a type-definable set.
\end{lem}
\begin{proof}
  All we need to say is that the tuple of parameters does indeed define a (finitely, or, in the continuous case, approximately finitely) consistent type, which is indeed a type-definable property.
\end{proof}

\begin{lem}
  \label{lem:BijectionUniformCB}
  Let $\Cb$ be a uniform canonical base map in the sort $\bar x$, and let $f$ be definable function (without parameters) defined on $\img \Cb$, into some other possibly infinite sort (this is equivalent to requiring that the graph of $f$ be type-definable).
  Assume furthermore that $f$ is injective.
  Then $\Cb' = f \circ \Cb$ is another uniform canonical base map.
  Moreover, every uniform canonical base map can be obtained from any other in this manner.
\end{lem}
\begin{proof}
  The main assertion follows from the fact that if $f$ is definable and injective and $d\varphi(\bar y,Z)$ is a formula then $d\varphi\bigl( \bar y,f^{-1}(W) \bigr)$ is also definable by a formula on the image of $f$.
  For the moreover part, given two uniform canonical base maps $\Cb$ and $\Cb'$, the graph of the map $f\colon \Cb(p) \mapsto \Cb'(p)$ is type-definable (one canonical base has to give rise to the same definitions as the other, and this is a type-definable condition), so $f$ is definable.
\end{proof}

Thus, in the same way that a canonical base for a type is exactly anything which is inter-definable with another canonical base for that type, a uniform canonical base is exactly anything which is uniformly inter-definable with another uniform canonical base.
A consequence of this (and of existence of uniformly canonical bases) is that in results such as the following the choice of uniform canonical bases is of no importance.

\begin{ntn}
  When $\cM$ is a model and $\bar a$ a tuple in some elementary extension, we write $\Cb(\bar a/M)$ for $\Cb(p)$ where $p = \tp(\bar a/M)$.
\end{ntn}

\begin{lem}
  \label{lem:DefinableFunctionUniformCB}
  Let $\bar z = f(\bar x,\bar y)$ be a definable function in $T$ (say without parameters), possibly partial, and let $\Cb$ be uniform.
  Then the map $f^\Cb\bigl( \Cb(\bar a/M), \bar b \bigr)
  = \Cb\bigl( f(\bar a,\bar b) / M \bigl)$ is definable as well for $(\bar a,\bar b) \in \dom f$, $\bar b \in M$, uniformly across all models of $T$.
  In case $f$ is definable with parameters in some set $A$, so is $f^\Cb$, uniformly across all models containing $A$.
\end{lem}
\begin{proof}
  For the first assertion, it is enough to observe that we can define $\tp\bigl( f(\bar a,\bar b)/M \bigr)$ by
  \begin{gather*}
    \varphi\bigl( f(\bar a,\bar b), \bar c \bigr) \quad = \quad d\psi\bigl( \bar b\bar c, \Cb(\bar a/M) \bigr),
  \end{gather*}
  where $\psi(\bar x,\bar y\bar z)
  = \varphi\bigl( f(\bar x,\bar y), \bar z \bigr)$.
  The case with parameters follows.
\end{proof}

\begin{lem}
  \label{lem:UniformlyContinuousUniformCB}
  Let $\Cb$ be a uniform canonical base map, say on the sort of $n$-tuples, into some infinite sort, and let $\Cb(p)_i$ denote its $i$th coordinate.
  Then the map $\bar a \mapsto \Cb(\bar a/M)_i$ is uniformly continuous, and uniformly so regardless of $\cM$.
\end{lem}
\begin{proof}
  For a uniform canonical base map constructed from a uniform definition as discussed before \fref{lem:ExistUniformCB} this follows from the fact that formulae are uniformly continuous.
  The general case follows using \fref{lem:BijectionUniformCB} and the fact that definable functions are uniformly continuous.
\end{proof}

\begin{rmk}
  The notion of a uniform canonical base map can be extended to simple theories, and the same results hold.
  Of course, canonical bases should then be taken in the sense of Hart, Kim and Pillay \cite{Hart-Kim-Pillay:CoordinatisationAndCanonicalBases}, and one has to pay the usual price of working with hyper-imaginary sorts.
\end{rmk}

Now the question we asked earlier becomes
\begin{qst}
  Let $T$ be a stable theory.
  Find a \emph{natural} uniform canonical base map for $T$.
  In particular, one may want the image to be in the home sort, or in a restricted family of imaginary sorts.
\end{qst}

Usually we shall aim for the image to lie in the home sort, plus the sort $\{T,F\}$ in the case of classical logic, or $[0,1]$ in the case of continuous logic.

\begin{exm}
  \label{exm:HilbertSpaceCB}
  Let $T = IHS$, the theory of infinite dimensional Hilbert spaces, or rather, of unit balls thereof (from now on we shall tacitly identify Banach space structures with their unit balls).

  The ``folklore'' canonical base for a type $p = \tp(\bar v/E)$ is the orthogonal projection $P_E(\bar v)$.
  Indeed, by quantifier elimination it will be enough to show that $p(\bar x)$ admits a $\varphi$-definition over $P_E(\bar v)$ for every formula of the form $\varphi(\bar x,y) = \| \sum \lambda_i x_i + y \|^2$, for any choice of scalars $\bar \lambda$.
  We then observe that for all $u \in E$ we have
  \begin{gather*}
    \varphi(\bar v,u) = \left\| \sum \lambda_i v_i \right\|^2 - \left\| \sum \lambda_i P_E(v_i) \right\|^2 + \left\| \sum \lambda_i P_E(v_i) + u \right\|^2.
  \end{gather*}
  The second and third terms are definable from $u$ using $P_E(\bar v)$ as parameters, while the first term is simply a constant which does not depend on $u$.

  Since $P_E(\bar v)$ alone does not allow us to recover $\| \sum \lambda_i v_i \|^2$, and therefore does not allow us to recover $p$ either, this canonical base is not uniform.
  By adding the missing information (namely, the inner product on the $v_i$, in the sort $[-1,1]$) we obtain a uniform canonical base:
  \begin{gather*}
    \Cb(\bar v/E) =
    \bigl( P_E(v_i), \langle v_i,v_j \rangle \bigr)_{i,j < n}.
  \end{gather*}
\end{exm}

This example, where we take a canonical base which is not uniform and make it uniform merely by adding information in a \emph{constant sort} (namely, $\{T,F\}$ in classical logic, or $[0,1]$ in continuous logic) is a special case of the following.

\begin{dfn}
  Say that a canonical base map is \emph{weakly uniform} if it can be obtained from a uniform map by composition with a definable function (which need not necessarily be injective, so the resulting canonical base need not suffice to recover the type uniformly -- compare with \fref{lem:BijectionUniformCB}).
\end{dfn}

For example, in the case of Hilbert spaces discussed above, the canonical base map $\tp(\bar v/E) \mapsto P_E(\bar v)$ is weakly uniform.

\begin{prp}
  \label{prp:WeaklyUniformCB}
  Let $\Cb$ be any uniform canonical base map, and let us write its target sort as $Z_0 \times Z_1$, where $Z_1$ is a power of the constant sort.
  Let $\Cb_0$ be the restriction to the sort $Z_0$.
  Then $\Cb_0$ is a weakly uniform canonical base map.
  Conversely, every weakly uniform canonical base map can be obtained in this fashion.
\end{prp}
\begin{proof}
  The main assertion is quite immediate, and it is the converse which we need to prove.
  Let $\Cb_0$ be a weakly uniform canonical base map on a sort $\bar x$, with target sort $Z_0$.
  By definition, it is of the form $f \circ \Cb'$, where $\Cb'$ is a uniform canonical base map with target sort $W$ and $f\colon \img \Cb' \to Z_0$ is definable.
  Let $\Phi$ be the set of all formulae $\varphi(\bar x,W)$.
  For every such formula, the value $\varphi\bigl( \bar a, \Cb'(\bar a/M) \bigr)$ is uniformly definable from $\Cb'(\bar a/M)$, call it $g_\varphi\bigl( \Cb'(\bar a/M) \bigr)$, and let $g = (g_\varphi)_{\varphi \in \Phi}$.
  Then $(f,g) \colon \img \Cb' \to Z_0 \times Z_1$ is definable, and $Z_1$ is a power of the constant sort.
  If we show that $(f,g)$ is injective then, by \fref{lem:BijectionUniformCB}, we may conclude that $\Cb = (f,g) \circ \Cb'$ is the desired uniform canonical base map.

  So let us consider a model $\cM$ and two tuples $\bar a$ and $\bar b$ in the sort $\bar x$, lying in some extension $\cN \succeq \cM$.
  Let $C = \Cb'(\bar a/M)$, $D = \Cb'(\bar b/M)$, and assume that $(f,g)(C) = (f,g)(D)$.
  Then $g(C) = g(D)$ means that $\bar a C \equiv \bar b D$.
  Since also $f(C) = f(D)$, we have $\bar a f(C) \equiv \bar b f(C)$, i.e., $\bar a \equiv_{f(C)} \bar b$.
  Finally, by hypothesis, $f(C)$ is a canonical base for both types, whence $\bar a \equiv_M \bar b$
  and therefore $C = D$.
  This completes the proof.
\end{proof}

Thus our question can be restated as
\begin{qst}
  Let $T$ be a stable theory.
  Find a natural weakly uniform canonical base map for $T$ with image in the home sort.
\end{qst}

Unfortunately, the canonical bases mentioned above for probability algebras and $L_p$ lattices are not even weakly uniform, so we cannot apply \fref{prp:WeaklyUniformCB} and the problem of finding uniform canonical bases requires some new ideas.

\section{Uniform canonical bases in atomless probability spaces}
\label{sec:RV}

The easier of the two ``interesting cases'' is that of atomless probability algebras.
There is a caveat, though, namely a uniform canonical base exists not in the home sort of events (see \fref{cor:PrNoUniformCanonicalBase}), but in the (still mathematically natural) imaginary sort of $[0,1]$-valued random variables.

In \cite{BenYaacov:RandomVariables} we define the theory $ARV$ of atomless spaces of $[0,1]$-valued random variables, in the language $\{0,\neg,\half,\vee,\wedge\}$, where $0$ is a constant, $\neg X = 1-X$ and $\half X$ are unary, the lattice operations $\vee$ and $\wedge$ are binary, and integration is recovered as $\bE[X] = d(X,0)$.
We show there that $ARV$ is bi-interpretable with $APr$, the theory of atomless probability algebra: to a probability algebra $\sA$ one associates the space $\cM = L^1(\sA,[0,1])$ of $[0,1]$-valued $\sA$-measurable random variables, from which one can recover $\sA = \sigma(M)$.
In addition, $ARV$ is $\aleph_0$-stable, eliminates quantifiers, and admits definable continuous calculus: if $\tau\colon [0,1]^n \to [0,1]$ is any continuous function then the map $\bar X \mapsto \tau(\bar X)$ is definable.
Thus, in order to construct uniform canonical bases for $APr$ in the sort of $[0,1]$-valued random variables, it will be enough (and in a sense, better) to show that $ARV$ admits uniform canonical bases in the home sort.

For models $\cM \preceq \cN \vDash ARV$ and $X \in N$ let us write $\bE[X|M]$ for the conditional expectation $\bE[X|\sigma(M)]$, itself a member of $M$.
For tuples $\bar X \in N^n$ and $\bar k \in \bN^n$, let $\bar X^{\bar k} = \prod X_i^{k_i}$.

\begin{lem}
  \label{lem:RVCondExp}
  Let $\cM \preceq \cN \vDash ARV$ and let $\bar X \in N^n$, $\bar k \in \bN^n$.
  Then $\bE[\bar X^{\bar k}|M]$ is uniformly definable from $\Cb(\bar X/M)$.
  Also, for every $\bar \ell \in \bN^m$, $\bE[ \bar X^{\bar k} \bar Y^{\bar \ell} ]$ is uniformly definable from $\bE[\bar X^{\bar k} | M]$ and $\bar Y \in M^m$.
\end{lem}
\begin{proof}
  By the definable continuous calculus, the function $(\bar X,Y) \mapsto |\bar X^{\bar k} - Y|^2$ is uniformly definable, and by \fref{lem:DefinableFunctionUniformCB} the predicate $\| \bar X^{\bar k} - y\|_2$ is uniformly definable for $y \in M$ from $\Cb(\bar X/M)$.
  For $Y \in M$ we have
  \begin{gather*}
    Y = \bE[\bar X^{\bar k}|M]
    \qquad \Longleftrightarrow \qquad
    \| \bar X^{\bar k} - Y\|_2
    =
    \inf_y \| \bar X^{\bar k} - y\|_2,
  \end{gather*}
  where the infimum is taken in $\cM$.
  Thus the graph of the function $\Cb(\bar X/M) \mapsto \bE[\bar X^{\bar k}|M]$ is type-definable in $\cM$, whence it follows that the function itself is definable, and uniformly so in all models of $ARV$.
  For the second assertion, use definable continuous calculus and the fact that for $\bar Y \in M^m$ we have $\bE[ \bar X^{\bar k} \bar Y^{\bar \ell} ] = \bE\bigl[ \bE[ \bar X^{\bar k} | M ] \cdot \bar Y^{\bar \ell} \bigr]$.
\end{proof}

\begin{thm}
  \label{thm:ARVUniformCanonicalBase}
  For $n$-types over models in $ARV$,
  \begin{gather*}
    \Cb(\bar X/M) = \bigl( \bE[\bar X^k|M] \bigr)_{\bar k \in \bN^n}
  \end{gather*}
  is a uniform canonical base (in the home sort).
\end{thm}
\begin{proof}
  By \fref{lem:RVCondExp}, this tuple is uniformly definable from any other uniform canonical base, so by \fref{lem:BijectionUniformCB} all that is left to show is that this tuple determines the type.
  By \fref{lem:RVCondExp} again, this tuple determines $\bE[ \bar X^{\bar k}\bar Y^{\bar \ell} ]$ for every $\bar Y \in M^m$ and $\bar \ell \in \bN^m$.
  By the Stone-Weierstaß Theorem, every continuous function $\tau\colon [0,1]^{n + m} \rightarrow [0,1]$ is uniformly approximated by polynomials, so $\bE[ \tau(\bar X,\bar Y)]$ is determined for every such $\tau$ and $\bar Y$, which, by quantifier elimination, is enough to determine $\tp(\bar X/M)$.
\end{proof}

\begin{cor}
  \label{cor:APrUniformCanonicalBase}
  For $n$-types over models in $APr$,
  \begin{gather*}
    \Cb(\bar A/\sA) = \left( \bP\left[\left. \bigwedge_{i \in s} A_i \right| \sA \right] \right)_{\emptyset \neq s \subseteq n}
  \end{gather*}
  is a uniform canonical base (in the sort of $[0,1]$-valued random variables), where $\bP[\cdot|\sA]$ denotes conditional probability.
\end{cor}
\begin{proof}
  Immediate.
\end{proof}

Let us now show that this cannot be improved.

\begin{lem}
  \label{lem:NoInjectiveRVToPr}
  There is \emph{no} $\emptyset$-definable injective map from a model $\cM \vDash ARV$ to $\sigma(M)^\kappa \times [0,1]^\lambda$ (for any cardinals $\kappa$, $\lambda$).
\end{lem}
\begin{proof}
  Assume there were one, say $(f,g)\colon M \rightarrow \sigma(M)^\kappa \times [0,1]^\lambda$, where $f$ goes to $\sigma(M)^\kappa$ and $g$ to $[0,1]^\lambda$.
  Let $A \in \sigma(M)$ have measure one half, with indicator function $\bone_A$, and for $0 \leq r \leq 1$ let $X_r = r  \bone_A \in M$.
  For $i < \kappa$, the event $f_i(X_r)$ must be definable from $X_r$, namely from $A$, and therefore must belong to the discrete set $\{0,1,A,\neg A\}$.
  Since the map $f_i$, being definable, must be continuous, we conclude that $r \mapsto f_i(X_r)$ must be constant, so $f(\bone_A) = f(0)$.

  Now let $B \in \sigma(M)$ be any other event of measure one half.
  Then $f(\bone_A) = f(0) = f(\bone_B)$, and by quantifier elimination $\bone_A \equiv \bone_B$, whereby $g(\bone_A) = g(\bone_B)$, so $(f,g)$ is not injective after all.
\end{proof}

\begin{cor}
  \label{cor:PrNoUniformCanonicalBase}
  The theory $APr$ does not admit a weakly uniform canonical base in its home sort (of events).
\end{cor}
\begin{proof}
  Let $\sA \vDash APr$, say the probability algebra of $\Omega$, and let $\sB$ be the probability algebra of $\Omega \times [0,1]$.
  Then $\sB \vDash APr$ as well, and identifying $A \in \sA$ with $A \times [0,1]$, we obtain an embedding $\sA \subseteq \sB$ which is elementary by quantifier elimination of $APr$.
  More generally $X \in L^1(\sA,[0,1])$, let $B_X = \{(\omega,r)\colon r \leq X(\omega)\} \in \sB$ (which is well defined, up to null measure).
  Thus, with $\Cb$ a uniform canonical base map as in \fref{cor:APrUniformCanonicalBase} (for types of singletons), we have $\Cb(B_X/\sA) = \bP[B_X|\sA] = X$, so $\img \Cb = L^1(\sA,[0,1])$.

  Now, assume there existed a weakly uniform canonical base map for types of singletons in $APr$ with values in the home sort.
  Then by \fref{prp:WeaklyUniformCB} there would exist a uniform canonical base map with values in the home sort and $[0,1]$, and by \fref{lem:BijectionUniformCB} it is of the form $f \circ \Cb$, where $f\colon \img \Cb \rightarrow \sA^\kappa \times [0,1]^\lambda$ is definable and injective.
  This (together with the bi-interpretability of $\sA$ with $L^1(\sA,[0,1])$) contradicts \fref{lem:NoInjectiveRVToPr}.
\end{proof}

\section{Uniform canonical bases in atomless $L_p$ lattices}
\label{sec:Lp}

Recall that $LpL$ denotes the theory of $L_p$ lattices for some fixed $p \in [1,\infty)$, and that $ALpL$ denotes the theory of atomless ones: their models are exactly those abstract Banach lattices which are isomorphic to a Banach lattice of the form $L_p(\Omega)$, where $\Omega$ is a measure space, atomless in the case of $ALpL$.
We may write the measure space more explicitly as $(\Omega,\Sigma,\mu)$, where $\Omega$ is a set of points, $\Sigma$ a $\sigma$-algebra on $\Omega$ and $\mu$ a $\sigma$-additive measure on $\Sigma$, atomless in the case of $ALpL$.

Stability, independence and related notions were studied for $ALpL$ by Berenstein, Henson and the author in \cite{BenYaacov-Berenstein-Henson:LpBanachLattices}.
Here we shall consider Banach lattices (as we do all Banach space structures) via their unit balls (from which the entire Banach lattice can be recovered), in the language $\{0,-,\half[x+y],|\cdot|\}$, where, for æsthetic reasons, we halve the distance: $d(x,y) = \|\half[x-y]\|$ (see \cite{BenYaacov:NakanoSpaces}).
The results of \cite{BenYaacov-Berenstein-Henson:LpBanachLattices} were given in a different, yet equivalent, formalism, so they hold fully in ours.
The theory $ALpL$ was shown to be $\aleph_0$-stable, and canonical bases of $1$-types were described as tuples of \emph{conditional slices} in the home sort (see Section~5 there).
Even though they are very natural invariants of a $1$-type, conditional slices are not uniform, or even weakly uniform, in the sense of the present paper.
Our aim here is to replace the conditional slices with a related object which does provide a uniform canonical base.

We start by quickly recalling the \emph{Krivine calculus} on Banach lattices (see also \cite{Lindenstrauss-Tzafriri:ClassicalBanachSpacesII}).
\begin{lem}
  Every lattice term $t(\bar x)$ defines a function $t\colon \bR^n \to \bR$ which is finitely piecewise affine, continuous, and $\bR^+$-homogeneous of degree one, by which we mean that $t(\alpha\bar x) = \alpha t(\bar x)$ for all $\alpha \geq 0$.

  In addition, an arbitrary function $\varphi\colon \bR^n \to \bR$ is continuous and $\bR^+$-homogeneous of degree one if and only if it can be approximated by lattice terms uniformly on every compact.
\end{lem}
\begin{proof}
  The first assertion, as well as the if part of the second, are clear.
  For the only if, let us assume that $\varphi$ is continuous and $\bR^+$-homogeneous of degree one.
  Then it is determined by its restriction to the unit sphere.
  Since every uniform approximation of $\varphi$ on the unit sphere yields a uniform approximation on the entire unit ball, it will be enough to show that lattice terms are dense in $\cC(S^{n-1},\bR)$.
  They obviously form a lattice there, so it will be enough to show that for every distinct $\bar x,\bar y \in S^{n-1}$, every $a,b \in \bR$ and every $\varepsilon > 0$, there is a lattice term $t$ such that $|t(\bar x)-a|, |t(\bar y)-b| < \varepsilon$.

  We may assume that $x_0 \neq y_0$.
  If $|x_0| = |y_0|$ then we may assume that $x_0 < 0 < y_0$ and define $t(\bar z) = \frac{a}{y_0} z_0^- + \frac{b}{y_0} z_0^+$.
  Otherwise, we may assume that $|x_0| < |y_0|$, in which case the opposite inequality must hold for some other coordinate, say $|y_1| < |x_1|$.
  Then $x_1y_0 - x_0y_1 \neq 0$ and we may define
  \begin{gather*}
    t(\bar z)
    =
    \frac{bx_1-ay_1}{x_1y_0 - x_0y_1}z_0
    + \frac{ay_0-bx_0}{x_1y_0 - x_0y_1}z_1.
  \end{gather*}
  Either way, $t(\bar x) = a$ and $t(\bar y) = b$, which is even better than what we needed.
\end{proof}

It is also a fact that if $t$ is a lattice term with bound $\|t\|$ on $[-1,1]^n$ then for any Banach lattice $E$ and every sequence $\bar f \in E^n$ one has $|t(\bar f)| \leq \|t\| \bigvee_i |f_i|$.
It follows that if $t_k \to \varphi$ uniformly on $[-1,1]^n$ then $t_k(\bar f)$ converges in norm to a limit $\varphi(\bar f)$ which does not depend on the choice of converging lattice terms, and at a rate which only depends on the sequence $t_k$ and on $\| \bigvee_i |f_i| \|$.
It follows that the map $\varphi\colon E^n \to E$ is uniformly definable across all (unit balls of) Banach lattices.

\begin{conv}
  For $\alpha > 0$ we extend $x \mapsto x^\alpha$ to the whole real line by $(-x)^\alpha = -x^\alpha$ (so $(-7)^2 = -49$).
\end{conv}

\begin{lem}
  \label{lem:LpLqDefinition}
  For every $p,q \in [1,\infty)$ the theories $LpL$ and $LqL$ are quantifier-free bi-definable.
  More exactly, if $E = L_p(\Omega)$ and $F = L_q(\Omega)$ then we may identify their respective underlying sets via the bijection $f \mapsto f^{\frac{p}{q}}$.
  Under this identification, each Banach lattice structure is quantifier-free definable in the other.

  Moreover, if $q \in (1,\infty)$ and $q' = \frac{q}{q-1}$ is its conjugate exponent, then the duality pairing $\langle f,g \rangle = \int fg$ on $L_q \times L_{q'}$ is definable in $L_p$.
\end{lem}
\begin{proof}
  The map $\theta\colon f \mapsto f^{\frac{p}{q}}$ clearly is a bijection between $L_p(\Omega)$ and $L_q(\Omega)$.
  For any definable map or predicate in $L_q(\Omega)$, call it $\tau$, let $\tau^{L_q}$ denote its pre-image in $L_p(\Omega)$ by $\theta$.
  Thus, for example, $\|x\|^{L_q} = \|x\|^{\frac{p}{q}}$, so $\theta$ is also a bijection between the respective unit balls.
  Also, $0^{L_q} = 0$, and similarly for $-$, and $|\cdot|$ (namely, they commute with $\theta$).
  For $\half[x+y]$, we have $(\half[x+y])^{L_q} = \left( \half[ x^{p/q} + y^{p/q} ] \right)^{\frac{q}{p}}$: as a map $\bR^2 \rightarrow \bR$, this is continuous and $\bR^+$-homogeneous of degree one, and therefore uniformly definable in all Banach lattices, and in particular in models of $LpL$.
  Finally, $d(x,y)^{L_q} = \|\half[ x - y ]\|^{L_q} = \bigl\| (\half[x-y])^{L_q} \bigr\|^{\frac{p}{q}}$, and we already know that $(\half[x-y])^{L_q}$ is definable.
  It follows that the predicate $f \mapsto \int f^p$ is definable in $L_p(\Omega)$: when $p = 1$ this is merely $\|f^+\| - \|f^-\|$, where $f^+ = \half[|f| + f]$ and $f^- = \half[|f| - f]$, and for the general case we use the bi-definability proved above with $q = 1$.
  By the same arguments, $(x,y) \mapsto x^{\frac{1}{q}} y^{\frac{1}{q'}}$ is definable in $L_p(\Omega)$.
  Given $f,g \in L_p(\Omega)$ which represent $f^{\frac{p}{q}} \in L_q(\Omega)$ and $g^{\frac{p}{q'}} \in L_{q'}(\Omega)$, respectively, we have $\left\langle f^{\frac{p}{q}},g^{\frac{p}{q'}} \right\rangle = \left( \int f^{\frac{1}{q}} g^{\frac{1}{q'}} \right)^p$, which is definable from $f^{\frac{1}{q}} g^{\frac{1}{q'}}$ and therefore from $(f,g)$, concluding the proof.
\end{proof}

Recall from \cite{BenYaacov-Berenstein-Henson:LpBanachLattices} that for any embedding $E \subseteq F$ of $L_p$ lattices there are presentations of $E$ and $F$ as $L_p(\Omega_1,\Sigma_1,\mu_1) \subseteq L_p(\Omega_2,\Sigma_2,\mu_2)$, such that $\Omega_1 \subseteq \Omega_2$, $\Sigma_1 \subseteq \Sigma_2$ and $\mu_2$ extends $\mu_1$.
For $g \in E$, $f \in F$, and presentation as above, we say that $g = \bE[f|\Sigma_1]$ if $\int_A g = \int_A f$ for all $A \in \Sigma_1$ of finite measure.
Uniqueness of $g$ is clear, and existence is fairly standard (at least under some mild hypotheses on the measure spaces in question).
It is shown in \cite{BenYaacov-Berenstein-Henson:LpBanachLattices} that, moreover, $\bE[f|\Sigma_1]$ does not depend on the choice of presentation (i.e., of measure spaces), and is therefore denoted $g = \bE[f|E]$, the \emph{conditional expectation} of $f$ with respect to $E$, which will be an essential tool for this section.
We shall see below (\fref{lem:LpUniformCondExp}) that for $p > 1$, the Krivine calculus provides us with a more elementary argument showing that the conditional expectation depends only on $\tp(f/E)$ (and not on the presentation), and moreover definably so (an even more elementary argument shows independence of the presentation when $p = 1$, see \fref{rmk:LpOneConditionalExpectation}).

Let us recall from \cite{BenYaacov-Berenstein-Henson:LpBanachLattices} a few facts regarding types in $ALpL$.
Recall first that $f,g$ in a Banach lattice are called \emph{orthogonal}, in symbols $f \perp g$, if $|f| \wedge |g| = 0$, and if $E \subseteq F$ are two Banach lattices then $E^\perp = \{f \in F\colon f \perp E\} = \{f \in F\colon f \perp g \text{ for all } g \in E\}$ is a sub-lattice of $F$.
Recall also that if $E \subseteq F$ are $L_p$ lattices then $F = E^{\perp\perp} \oplus E^\perp$, so each $f \in F$ can be written uniquely as $f = f\rest_E + f\rest_{E^\perp}$, where $f\rest_{E^\perp} \perp E$ and $f\rest_E \perp E^\perp$.

\begin{fct}
  \label{fct:LpTypeCharacterisation}
  By \cite[Lemma~3.8]{BenYaacov-Berenstein-Henson:LpBanachLattices}, $\tp(f/E)$ is determined by $\tp(f\rest_E/E)$ together with $\|f\rest_{E^\perp}^\pm\|$, or equivalently, together with $\|f^\pm\|$.
  If $E = L_p(\Omega_1,\Sigma_1,\mu_1) \subseteq F = L_p(\Omega_2,\Sigma_2,\mu_2)$ is a concrete representation as above then $f\rest_E = f \mathbf{1}_{\Omega_1}$ and $f\rest_{E^\perp} = f \mathbf{1}_{\Omega_2 \setminus \Omega_1}$, and by \cite[Proposition~3.7]{BenYaacov-Berenstein-Henson:LpBanachLattices}, $\tp(f\rest_E/E)$ is determined by the $\Sigma_1$-conditional distribution of $f\rest_E$, namely by the map $t \mapsto \bP[f > t|\Sigma_1]$.
\end{fct}

We say that $F$ is \emph{atomless over $E$} if for every $g \in E$, $g \geq 0$, there is $f \in F$, $0 \leq f \leq g$, such that $f \perp g - f$ and $\|f\| = \|g\|/2$, or equivalently, if for any concrete presentation as above, $\mu_2$ is atomless over $\Sigma_1$.
If this is the case then, by virtue of the above, together with \cite[Fact~2.15]{BenYaacov-Berenstein-Henson:LpBanachLattices}, for any $f$ in an extension of $E$, $\tp(f\rest_E/E)$ is realised in $F$.
If, in addition, $E^\perp \subseteq F$ is non trivial, then any $1$-type (or indeed, $n$-type, but we shall not require this) over $E$ is realised in $F$.

Let $\Omega$ be a measure space, and let $\sA(\Omega)$ be the algebra (or more precisely, the relatively complemented distributive lattice) of finite measure sets modulo null measure sets, which, with some abuse of terminology, we shall call the \emph{measure algebra} of $\Omega$.
Let $\Omega'$ be the Stone space of $\sA(\Omega)$, an extremally disconnected locally compact space, equipped with the unique regular Borel measure which associates to each compact open set the (finite) measure of the corresponding measurable subset of $\Omega$, which exists by Carathéodory's Theorem.

\begin{fct}
  \label{fct:StoneSpaceOmega}
  With $\Omega$ and $\Omega'$ as above, $\sA(\Omega')$ is canonically isomorphic to $\sA(\Omega)$, yielding a canonical identification $L_p(\Omega) \cong L_p(\Omega')$, so we may assume that $\Omega = \Omega'$.
  Such a space has the particular property that every $g \in L_p(\Omega)$ has a unique continuous representative, so it makes sense to speak of the value $g(\omega) \in [-\infty,\infty]$ at some $\omega \in \Omega$, which is finite for almost all $\omega$.

  Algebraic and lattice operations, and therefore the lattice order, are reflected in these continuous representatives.
  If $A \subseteq L_p(\Omega)$ is countable (or separable) then $A$ is bounded in $L_p(\Omega)$ if and only if the point-wise supremum $g(\omega) = \sup_{f \in A} f(\omega)$ is in $L_p$, in which case it is equal almost everywhere to the continuous representative of $\sup A$.
\end{fct}
\begin{proof}
  The construction is straightforward.
  To see that it has the desired properties, it is enough to consider restrictions to finite measure sub-spaces, in which case this is shown in \cite[Theorem~1.9 and Proposition~1.14] {BenYaacov:SchroedingersCat}.
  For a more elaborate discussion, see \cite[Chapter~32]{Fremlin:MeasureTheoryVol3}.
\end{proof}

This spares us the headache of choosing representatives: for example, given a convex function $x \mapsto \Psi_f(x) \in L_p(\Omega)$, it is automatically true that $x \mapsto \Psi_f(x)(\omega)$ is convex for each $\omega$.

Let us fix $E \vDash ALpL$ and a type $q = \tp(f/E)$ for some $f$ in an extension of $E$.
Let us also fix a concrete representation $E = L_p(\Omega)$, with $\Omega$ as in \fref{fct:StoneSpaceOmega}, and define \hbox{$E' = L_p\bigl( ( \Omega \cup \{+, -\} ) \times [0,1] \bigr)$}, where the interval is taken with the Lebesgue measure, and each of $\pm$ has measure one.
Let $\pi_0\colon \Omega \times [0,1] \rightarrow \Omega$ denote the projection, and identify a function $g \colon \Omega \rightarrow \bR$ with $g \circ \pi_0\colon \Omega \times [0,1] \rightarrow \bR$, as well as with its extension by zeroes to $( \Omega \cup \{\pm\} ) \times [0,1]$, yielding a canonical embedding $E \subseteq E'$, which, by quantifier elimination for $ALpL$, is elementary.
By the previous paragraph, we may assume that $f \in E'$.
The conditional expectation of $f$ with respect to $\Omega$ (namely, to its Borel algebra) is
\begin{gather*}
  \bE[f|\Omega](\omega) = \int_0^1 f(\omega,t) \, dt,
\end{gather*}
and $\| \bE[f|\Omega] \| \leq \|f\|$ by Jensen's inequality and convexity of $x \mapsto x^p$ (for $f \geq 0$ the integral is finite almost everywhere, by Fubini's Thoerem and Jensen's inequality, so it is defined almost everywhere for all $f$).

\begin{lem}
  \label{lem:LpUniformCondExp}
  Assume that $p > 1$, and let $E \vDash ALpL$, $q \in \tS_1(E)$, and $f \vDash q$ in $E'$ as above.
  Then the conditional expectation $\bE[f|\Omega]$ defined above depends only on $q$, that is to say that for any presentation $E = L_p(\Omega') \subseteq L_p(\Omega'') = F$, and any $f' \in F$, if $f' \vDash q$ then $\bE[f'|\Omega'] = \bE[f|\Omega]$.
  Moreover, for any uniform canonical base map $\Cb$, the map $\Cb(f/E) \mapsto \bE[f|E]$ (where $\bE[f|E]$ now denotes $\bE[f|\Omega]$) is definable in $ALpL$.
  More generally, if $t(x,\bar y)$ is any lattice term, then the map $\bigl( \Cb(f/E),\bar g \bigr)
  \mapsto
  \bE\bigl[ t(f,\bar g) | E \bigr]$, where $\bar g \in E$, is definable.
\end{lem}
\begin{proof}
  Clearly, if $g \in E$ satisfies that $\langle f,h \rangle = \langle g, h \rangle$ for all $h \in E^* = L_{p'}(\Omega)$, where $p' = \frac{p}{p-1}$ then $g = \bE[f|\Omega]$, and it is fairly standard to check that the converse holds as well.
  It follows that $g = \bE[f|\Omega]$ if and only if $\int f h^{p-1} = \int g h^{p-1}$ for all $h \in E$, or equivalently, for all $h$ in the unit ball of $E$.
  The map $(x,y) \mapsto x^{\frac{1}{p}} y^{\frac{p-1}{p}}$ is uniformly definable by the Krivine calculus, so $(x,y) \mapsto \| x^{\frac{1}{p}} y^{\frac{p-1}{p}} \|^p = \int |x| |y|^{p-1}$ is a definable predicate.
  Separating into positive and negative parts we see that $\varphi(x,y) = \int x y^{p-1}$ is a definable predicate as well.
  Together with our characterisation of $\bE[f|\Omega]$, this implies that it indeed only depends on $\tp(f/E)$ (in fact, all of this holds even when $E$ is a mere model of $LpL$, not necessarily atomless, in which case $\bE[f|\Omega]$ only depends on the quantifier-free type of $f$ over $E$).

  From now on we write $\bE[f|E]$ for $\bE[f|\Omega]$.
  With $\varphi$ as above, in $E$ we have $x = \bE[f|E]$ if and only if
  \begin{gather*}
    \sup_z \left|
      \varphi(x,z) - d\varphi\bigl( z, \Cb(f/E) \bigr)
    \right| = 0,
  \end{gather*}
  and this remains true if we restrict $z$ to the unit ball.
  Therefore the graph of $\Cb(f/E) \mapsto \bE[f|E]$ is type-definable, and moreover, given a bound on $\|f\|$ we have a bound on $\| \bE[f|E] \|$ (so, in a sense, we know in which ``sort'' $\bE[f|E]$ lives), so the map is definable.
  The case of a lattice term follows by \fref{lem:DefinableFunctionUniformCB}.
\end{proof}

\begin{rmk}
  \label{rmk:LpOneConditionalExpectation}
  This argument does not work for $p = 1$, and we shall see below that for $p = 1$ the definability of $\Cb(f/E) \mapsto \bE[f|E]$ fails.
  However, when $p = 1$ we have $\int f = \|f^+\| - \|f^-\|$, and if $h \in E$ is positive, say with support $A$ in some presentation, then $\int f\rest_A = \lim_{N \rightarrow \infty} (f \wedge Nh) \vee (-Nh)$, so again $\bE[f|E]$ only depends on $\tp(f/E)$.
\end{rmk}

Given $f$ in some extension of $E$, let $f_0 = \bE\bigl[ |f| \big| E \bigr] \in E$, and consider the function $\Psi_f\colon \bR \rightarrow E$, $\Psi_f(x) = \bE[ x f_0 \dotminus f|E]$.
It is convex, and therefore admits one-sided derivatives
\begin{gather*}
  D^+_x \Psi_f = \lim_{\varepsilon \searrow 0} \frac{\Psi_f(x + \varepsilon) - \Psi_f(x)}{\varepsilon} = \inf_{\varepsilon > 0} \frac{\Psi_f(x + \varepsilon) - \Psi_f(x)}{\varepsilon},
\end{gather*}
and similarly
\begin{gather*}
  D^-_x \Psi_f = \sup_{\varepsilon > 0} \frac{\Psi_f(x - \varepsilon) - \Psi_f(x)}{-\varepsilon}.
\end{gather*}
The function $\Psi_f$ is determined by $\tp(f/E)$, and by \fref{lem:LpUniformCondExp}, for $p > 1$ it is even uniformly definable from the canonical base.
Conversely,

\begin{conv}
  \label{conv:StandardRepresentation}
  Given $f$ in an extension of $E$ and $f_0 = \bE\bigl[ |f| \big| E \bigr]$ as above, one can always choose a concrete representation $E = L_p(\Omega)$ such that $f_0 = \bone_A$ is an indicator function.
  We may moreover assume that $\Omega$ is as in \fref{fct:StoneSpaceOmega}, in which case $A$ is a compact open set.
  We fix such a presentation for the rest of the section.

  We may still assume that \hbox{$f \in E' = L_p\bigl( ( \Omega \cup \{+, -\} ) \times [0,1] \bigr)$}, as in the discussion preceding \fref{lem:LpUniformCondExp}.
\end{conv}

\begin{lem}
  \label{lem:PsiDeterminesType}
  Under \fref{conv:StandardRepresentation}, the conditional distribution of $f\rest_E$ with respect to $\Omega$ is determined by
  \begin{gather*}
    \bP[f\rest_E \leq x|\Omega] \cdot \bone_A = \bP[f \leq xf_0|\Omega] \cdot f_0 = D^+_x \Psi_f, \qquad f\rest_E \cdot (1 - \bone_A) = 0,
  \end{gather*}
  In particular, the triplet $(\Psi_f,\|f^\pm\|)$, determines $\tp(f/E)$.
\end{lem}
\begin{proof}
  Indeed, for $\omega \in A$ we have
  \begin{gather*}
    \Psi_f(x)(\omega) = \mu_{[0,1] \times \bR}\Bigl\{ (t,s) \colon f(\omega,t) \leq s \leq x \Bigr\}
  \end{gather*}
  and
  \begin{align*}
    D^+_x \Psi_f(\omega)
    & = \inf_\varepsilon \frac{1}{\varepsilon} \mu_{[0,1] \times \bR} \Bigl\{ (t,s)\colon f(\omega,t) \vee x \leq s \leq x + \varepsilon \Bigr\}
    \\
    & = \mu_{[0,1]} \Bigl\{ t\colon f(\omega,t) \leq x \Bigr\} = \bP[f\rest_E \leq x|\Omega](\omega).
  \end{align*}
  The second assertion now follows from \fref{fct:LpTypeCharacterisation}.
\end{proof}

A useful tool in the study of convex functions is the \emph{Legendre transform}, which allows to pass between the spaces of points and of slopes.
For a function $\varphi\colon \bR \rightarrow \bR \cup \{\infty\}$, not identically $\infty$, we define
\begin{gather*}
  \varphi^*(t) = \sup_{x \in C_\varphi} tx - \varphi(x) \in \bR \cup \{\infty\}, \qquad \text{ where } C_\varphi = \{x \in \bR\colon \varphi(x) < \infty\}.
\end{gather*}
The function $\varphi^*$ is always convex and continuous on $\overline{C_{\varphi^*}}$ (namely at the boundary, since, being convex, it is automatically continuous in the interior).
If $\varphi^*$ is not identically $\infty$, which is always the case when $\varphi$ is convex, we can define the double Legendre transform $\varphi^{**}$, which is the greatest function below $\varphi$ which is convex and continuous on $\overline{C_\varphi}$, .
If $\varphi$ is convex and continuous on $\overline{C_\varphi}$ then $\varphi = \varphi^{**}$.
In this case, we have for every $x \in C_\varphi$ and $t \in C_{\varphi^*}$:
\begin{gather}
  \label{eq:LegendreTransformAttained}
  \varphi(x) + \varphi^*(t) = tx \quad \Longleftrightarrow \quad D^-_t \varphi^* \leq x \leq D^+_t \varphi^* \quad \Longleftrightarrow \quad D^-_x \varphi \leq t \leq D^+_x \varphi.
\end{gather}
(At the boundary of the domain, the derivatives may be $\pm\infty$, and are always finite in the interior.)
This can be generalised to functions on spaces of higher dimension, with $tx$ replaced with a duality pairing, see for example in Rockafellar \cite{Rockafellar:ConvexAnalysis}.

Working with a presentation $E = L_p(\Omega)$ as per \fref{conv:StandardRepresentation}, we can deduce similar properties for the lattice-valued $\Psi_f$.

\begin{lem}
  \label{lem:PsiDoubleLegendreTransform}
  Let $f \in F \supseteq E$, $f_0 = \bE\bigl[ |f| \big| E \bigr]$ and $\Psi_f(x) = \bE[x f_0 \dotminus f]$ be as above, and define
  \begin{gather*}
    \Psi_f^*(t) = \sup_x \, tx f_0 - \Psi_f(x), \qquad \Psi_f^{**}(x) = \sup_{t \in [0,1]} \, tx f_0 - \Psi_f^*(t).
  \end{gather*}
  Then $\Psi_f^*$ is finite, convex and continuous on $[0,1]$, with $\Psi_f^*(0) = 0$ and $\Psi_f^*(1) = \bE[f|E]$, it is infinite for $t \notin [0,1]$, and $\Psi_f = \Psi_f^{**}$.
\end{lem}
\begin{proof}
  We have $\|\Psi_f(x)\|_\infty \leq |x| + 1$ for $x \in \bR$, so for $\omega \in \Omega$ we may define a convex function $\psi_\omega\colon \bR \rightarrow \bR$ by $\psi_\omega(x) = \Psi_f(x)(\omega)$.
  For $x \leq y$ we have $\Psi_f(x) \leq \Psi_f(y) \leq \Psi_f(x) + (y-x) f_0$, and it follows that $\psi_\omega(x) \leq \psi_\omega(y) \leq \psi_\omega(x) + (y-x)$.
  Passing to the Legendre transform, this means that $\Psi_f^*(t) = \infty$ and $\psi_\omega^*(t) = \infty$ for all $\omega$ and $t \notin [0,1]$.
  For $t = 0,1$ we have
  \begin{gather*}
    \Psi_f^*(0) = \lim_{x \rightarrow -\infty} -\bE[x f_0 \dotminus f|E] = 0, \qquad \Psi_f^*(1) = \lim_{x \rightarrow \infty} \bE[x f_0 \wedge f|E] = \bE[f|E].
  \end{gather*}
  By convexity, $\Psi_f^*(t) \in E$ (i.e., is finite) for each $t \in [0,1]$, and by \fref{fct:StoneSpaceOmega}, $\psi_\omega^*(t) = \Psi_f^*(t)(\omega)$ for almost all $\omega$.
  Using the continuity of convex functions, we can exchange the quantifiers: for almost all $\omega$, $\psi_\omega^* = \Psi_f^*(\cdot)(\omega)$.
  It follows that for each $x$, for almost all $\omega$: $\psi_\omega^{**}(x) = \Psi_f^{**}(x)(\omega)$.
  Thus $\Psi_f(x) = \Psi_f^{**}(x)$, as desired.

  Let us show that $\Psi_f^*(t) \rightarrow 0$ as $t \searrow 0$ in $L_p$ (point-wise convergence almost everywhere is already known).
  For the positive part this just follows from convexity.
  For the negative part, for every $\varepsilon > 0$ there exists $s > 0$ such that $\bigl\| (-sf_0) \dotminus f \bigr\| < \varepsilon \|f_0\|$, so for $0 < \delta < \varepsilon/s$ we have
  \begin{gather*}
    \Psi_f^*(\delta) \geq -\delta s f_0 - \bE[(-sf_0) \dotminus f | E] \quad \Longrightarrow \quad \bigl\| \Psi_f^*(\delta)^- \bigr\| \leq 2\varepsilon \|f_0\|.
  \end{gather*}
  A similar argument works for $t \nearrow 1$.
\end{proof}

\begin{rmk}
  \label{rmk:Phi}
  If one is bothered with the presence of $f_0$ in the definition, one can also prove that
  \begin{gather*}
    \Psi_f^*(t) = \sup_{g \in E} \, tg - \bE[g \dotminus f|E], \qquad \bE[g \dotminus f|E] = \sup_{t \in [0,1]} \, tg - \Psi_f^*(t),
  \end{gather*}
  which looks formally closer to the (double) scalar Legendre transform.
  By the reasoning of \fref{lem:PsiDoubleLegendreTransform}, one can also prove a version of \fref{eq:LegendreTransformAttained}, namely
  \begin{gather*}
    \Psi^*(t) = tg - \bE[g \dotminus f] \quad \Longleftrightarrow \quad D^-_t \Psi_f^* \leq g \leq D^+_t \Psi_f^*.
  \end{gather*}
  This was indeed the path followed in an earlier version of this article, but the present seems more direct.

\end{rmk}

For $0 < t < 1$, let us define $f_t = D^-_t \Psi^*$, which determines an increasing sequence $(f_t) \subseteq E$.
Following \fref{conv:StandardRepresentation}, we may then define (measurably) $\hat f\colon ( \Omega \cup \{+, -\} ) \times [0,1] \rightarrow \bR$ by $\hat f(\omega,t) = f_t(\omega)$, $\hat f(+,\cdot) = \|f^+\rest_{E^\perp}\|$, $\hat f(-,\cdot) = -\|f^-\rest_{E^\perp}\|$ (ignoring $t = 0,1$).

\begin{lem}
  \label{lem:IncreasingRealisation}
  With $E \subseteq E'$, $f_t$ and $\hat f$ as above, we have $\hat f \in E'$, $\tp(f/E) = \tp(\hat f/E)$, and $\|f_t\| \leq \frac{\|f\|}{(t-t^2)^{1/p}}$ for all $0 < t < 1$.
\end{lem}
\begin{proof}
  Let us show that the conditional distribution of $\hat f\rest_E = \hat f \cdot \bone_\Omega$ with respect to the Borel algebra of $\Omega$, call it $\Sigma$, is the same as that of $f\rest_E$.
  Indeed, let $f_0 = \bone_A$.
  Then $\hat f \bone_{\Omega \setminus A} = 0 = f \bone_{\Omega \setminus A}$, while for (almost all) $\omega \in A$ we have by \fref{lem:PsiDeterminesType} and \fref{eq:LegendreTransformAttained}:
  \begin{gather*}
    \bP[\hat f \leq x|\Omega](\omega) = \sup \{ t \colon D^-_t \Psi_f^*(\omega) \leq x \} = \sup \{ t \colon t \leq D^+_x \Psi_f(\omega) \} = \bP[f \leq x|\Omega](\omega).
  \end{gather*}
  It follows that $\|\hat f\rest_E\|_p = \|f\rest_E\|_p < \infty$, so $\hat f \in E'$, and as in \fref{lem:PsiDeterminesType} we conclude that the types are the same.
  It follows that
  \begin{gather*}
    \|f\|^p \geq \|\hat f\rest_E\|^p = \int_0^1 \|f_t\|^p\,dt = \int_0^1 \|f_t^+\|^p + \|f_t^-\|^p\,dt.
  \end{gather*}
  Since $\|f_t^+\|$ (respectively, $\|f_t^-\|$) is increasing (respectively, decreasing) in $t$ we get
  \begin{gather*}
    \|f_t^+\|^p \leq \frac{\|f\|^p}{1-t}, \qquad \|f_t^-\|^p \leq \frac{\|f\|^p}{t},
  \end{gather*}
  whence the stated inequality.
\end{proof}

A comparison with \cite[Section~5]{BenYaacov-Berenstein-Henson:LpBanachLattices} yields that $f_t$ agrees with the conditional slice $\bS_{1-t}(f/E)$, as defined there, for all $t$ at which $\Psi_f^*$ is differentiable (namely, at which $D^+_t \Psi_f^* = D^-_t \Psi_f^* = f_t$), and in particular for almost all $t$.
The realisation constructed above is called there an ``increasing realisation'', and is shown to be unique when $f = f\rest_E$.
It follows that each $1$-type over $E$ has a unique realisation in $f \in E'$ such that $f(\omega,t)$ is increasing in $t$ on $\Omega \times [0,1]$, is a non negative constant on $\{+\} \times [0,1]$, and is a non positive constant on $\{-\} \times [0,1]$.
In other words, the map $\Psi_f^*$ (together with $\|f^\pm\|$) not only determines $\tp(f/E)$, but also provides the means for a fairly direct construction of a canonical realisation of the type.

The continuity of $\Psi_f^*$ implies that $\Psi_f^*(t) = \int_0^t f_s\, ds$.
Since $\Psi_f^*(1) = \int_0^1 f_s\, ds = \bE[f|E]$, we may think of $\Psi_f^*(t)$ as a ``part'' of the conditional expectation, (with complement $\bE[f|E] - \Psi_f^*(t) = \int_t^1 f_s\, ds$), calling it a \emph{partial conditional expectation}, denoted accordingly $\bE_t[f|E] = \Psi_f^*(t)$.
In particular, $\bE_0[f|E] = 0$ and $\bE_1[f|E] = \bE[f|E]$.

\begin{lem}
  \label{lem:LpUniformPartialCondExp}
  Assume that $p > 1$.
  Then $\Cb(f/E) \mapsto \bE_t[f|E]$ is definable in $ALpL$ for all $t \in [0,1]$.
\end{lem}
\begin{proof}
  For $t = 0,1$ this is already known by \fref{lem:LpUniformCondExp}.
  Since $\Cb(f/E) \mapsto \Cb(|f|/E)$ is also definable, $f_0 = \bE\bigl[ |f| \big| E \bigr]$ is uniformly definable from the canonical base.
  We may therefore assume that $0 < t < 1$, and use $f_0$ as a parameter.
  Similarly, for each $x$ the map $\Cb(f/E) \mapsto \Psi_f(x)$ is definable.

  First of all, let us fix $N > 0$, and let
  \begin{gather*}
    g_{N,n} = \bigvee_{|k| \leq n} tNkf_0/n - \Psi_f(Nk/n), \qquad h_N = \sup_{-N \leq x \leq N} \, tx f_0 - \Psi_f(x).
  \end{gather*}
  Then the map $\Cb(f/E) \mapsto g_{N,n}$ is definable (for each $n$), and $0 \leq h_N - g_{N,n} \leq 2N f_0/n$, so $g_{N,n} \rightarrow h_N$ at a rate which depends on $\|f\|$ and on $N$ alone.
  It follows that $\Cb(f/E) \mapsto h_N$ is definable.

  By \fref{eq:LegendreTransformAttained} we have $h_N(\omega) = \bE_t[f|E](\omega)$ when $|f_t(\omega)| \leq N$, and otherwise $|h_N(\omega) - \bE_t[f|E](\omega)| \leq 2$, since
  \begin{gather*}
    -f_0 \leq -\bE[f^-|E] \leq h_N \leq \bE_t[f|E] \leq t\bE[f|E] \leq f_0.
  \end{gather*}
  The measure of $\{\omega\colon |f_t(\omega)| > N\}$ vanishes as $N \rightarrow \infty$, and at a rate which depends only on $\|f_t\|$, which we already bounded in terms of $t$ and $\|f\|$.
  We conclude that $h_N \rightarrow \bE_t[f|E]$ at a rate which depends only on $\|f\|$ and $t$, which is enough.
\end{proof}

\begin{thm}
  \label{thm:LpNotOneUniformCB}
  For every $p \in (1,\infty)$ and every dense subset $D \subseteq (0,1)$ (e.g., $D = \bQ \cap (0,1)$), the tuple $\bigl( \|f^+\|, \|f^-\|, \bE_t[f|E] \bigr)_{t \in D}$ is a uniform canonical base for $\tp(f/E)$, and $\bigl( \bE_t[f|E] \bigr)_{t \in D}$, is a weakly uniform canonical base in the home sort.
\end{thm}
\begin{proof}
  By \fref{prp:WeaklyUniformCB} it is enough to prove the first assertion.
  We have already seen that $\Cb(f/E) \mapsto \bE_t[f|E]$ is definable, and clearly $\Cb(f/E) \mapsto \|f^\pm\|$ are.
  On the other hand, by \fref{lem:PsiDoubleLegendreTransform} the tuple $\bigl( \|f^+\|, \|f^-\|, \bE_t[f|E] \bigr)_{t \in D}$ determines $\Psi_f$ and therefore, by \fref{lem:PsiDeterminesType}, it determines $\tp(f/E)$.
\end{proof}

We have already observed that $\bE_t[f|E] \to 0$ as $t \to 0$ and $\bE_t[f|E] \to \bE[f|E]$ as $t \to 1$.
Moreover, for $p > 1$ (fixed) the rate of convergence depends uniformly on $\|f\|$.
Indeed, otherwise \fref{lem:LpUniformPartialCondExp} together with a compactness argument would yield a type (or a canonical base of a type, which is the same thing) for which convergence fails altogether.
On the other hand, for $p = 1$, consider for some $\varepsilon$ the case where $\mu(\Omega) = 1$ and $f_\varepsilon(x,t) = -\varepsilon^{-1}\mathbf{1}_{\Omega \times [0,\varepsilon]}$.
Then $\|f_\varepsilon\| = 1$ and $\bE_\varepsilon[f_\varepsilon|E] = -\mathbf{1}_\Omega$, also of norm one, so the rate of convergence is not uniform.
Thus \fref{lem:LpUniformPartialCondExp}, and therefore \fref{lem:LpUniformCondExp}, fail for $p = 1$.
This is essentially the only obstacle, and by keeping away from the endpoints of $[0,1]$ we do manage to get an analogue of \fref{thm:LpNotOneUniformCB} for $p = 1$.
For $0\leq t < s \leq 1$ let us define
\begin{gather*}
  \bE_{[t,s]}[f|E]
  =
  \bE_s[f|E] - \bE_t[f|E]
  =
  \int_t^s f_r\, dr.
\end{gather*}

\begin{lem}
  \label{lem:LpUniformIntervalCondExp}
  Let $p \in [1,\infty)$.
  Then for every $0<s<t<1$, the map $\Cb(f/E) \mapsto \bE_{[t,s]}[f|E]$ is definable.
\end{lem}
\begin{proof}
  For $p > 1$ this is already known, so we only need to deal with the case of $p = 1$.
  For each $q > 1$ we may apply the bidefinability of $AL1L$ and $ALqL$, and calculate $\bE_{[t,s]}[f|E]^{L_q} = \bE_{[t,s]}[f^{1/q}|E]^q$ uniformly from $\Cb(f/E)$ (as in the proof of \fref{lem:LpLqDefinition}, by $\bE_{[t,s]}[f|E]^{L_q}$ we mean $\bE_{[t,s]}[f|E]$ as calculated in the $L_q$ lattice defined in the $L_1$ lattice).
  It will be enough to show that as $q \to 1$, $\bE_{[t,s]}[f^{1/q}|E]^q \to \bE_{[t,s]}[f|E]$ at a rate which only depends on $\|f\|$.

  Concentrating on what happens over a single $\omega \in \Omega$, all functions in $E$ become constants, and we may identify $f$ with the function $f(t) = f_t$, which is increasing on $[0,1]$.
  If $f_0 = 0$ then everything is zero, so we may assume that $f_0 = \int_0^1 |f(t)|\, dt = 1$.
  Since $f$ is increasing, we must have $f(r) \in [-\frac{1}{t},\frac{1}{1-s}]$ for all $r \in [t,s]$, and by, say, dominated convergence,
  \begin{gather*}
    \left( \int_t^s f(r)^{\frac 1 q}\,dr \right)^q \rightarrow_{q \rightarrow 1} \int_t^s f(r)\,dr.
  \end{gather*}
  The space of all increasing functions $[t,s] \rightarrow [-\frac{1}{t},\frac{1}{1-s}]$ is compact in the topology of convergence in measure, which means that the convergence above is uniform in $f$, that is to say that for any desired $\varepsilon > 0$ there exists $q_0 = q_0(t,s,\varepsilon) > 1$ such that for all $1 < q < q_0$:
  \begin{gather*}
    \left| \left( \int_t^s f(r)^{\frac 1 q}\,dr \right)^q - \int_t^s f(r)\,dr \right| < \varepsilon.
  \end{gather*}
  Integrating we obtain
  \begin{gather*}
    \bigl| \bE_{[t,s]}[f^{1/q}|E]^q - \bE_{[t,s]}[f|E] \bigr| < \varepsilon \|f\|,
  \end{gather*}
  as desired.
\end{proof}

\begin{thm}
  \label{thm:LpUniformCB}
  For every $p \in [1,\infty)$ and every dense subset $D \subseteq (0,1)$ (e.g., $D = \bQ \cap (0,1)$), the tuple $\bigl( \|f^+\|, \|f^-\|, \bE_{[t,s]}[f|E] \bigr)_{t,s \in D, t<s}$ is a uniform canonical base for $\tp(f/E)$, and $\bigl( \bE_{[t,s]}[f|E] \bigr)_{t,s \in D, t<s}$ is a weakly uniform canonical base in the home sort.
\end{thm}
\begin{proof}
  As for \fref{thm:LpNotOneUniformCB}, since $\bE_s[f|E] = \lim_{t \rightarrow 0} \bE_{[t,s]}[f|E]$.
\end{proof}

We have thus produced (weakly) uniform canonical bases in the home sort for $1$-types in $ALpL$.
For $n$-types, we use the following general fact.

\begin{fct}
  The $n$-type $\tp(\bar f\rest_E/E)$ is determined by the $1$-types $\tp(\bar k \cdot \bar f/E)$, where $\bar k \cdot \bar f = \sum_i k_if_i$ and $\bar k$ varies over $\bZ^n$.
\end{fct}
\begin{proof}
  Indeed, this information determines $\tp(\bar t \cdot \bar f/E)$, and in particular $\tp(\bar t \cdot \bar f\rest_E/E)$ for all $\bar t \in \bQ^n$ and therefore for all $\bar t \in \bR^n$.
  Now apply \cite[Proposition~3.7]{BenYaacov-Berenstein-Henson:LpBanachLattices}.
\end{proof}

\begin{thm}
  Let $\Cb$ be a uniform canonical base map for $1$-types.
  Then
  \begin{gather*}
    \tp(\bar f/E)
    \mapsto
    \bigl( \Cb(\bar k \cdot \bar f/E), \tp(\bar f) \bigl)
    _{\bar k \in \bZ^n}
  \end{gather*}
  is a uniform canonical base map for $n$-types.
  (We may view $\tp(\bar f)$ as a sequence in $[0,1]$ via any embedding of $\tS_n(ALpL)$ in $[0,1]^{\aleph_0}$.)

  If $\Cb$ is a weakly uniform canonical base map for $1$-types then
  \begin{gather*}
    \tp(\bar f/E)
    \mapsto
    \bigl( \Cb(\bar k \cdot \bar f/E) \bigl)_{\bar k \in \bZ^n}
  \end{gather*}
  is a weakly uniform canonical base map for $n$-types.
\end{thm}
\begin{proof}
  For the first assertion, it is enough to show that $\tp(\bar f/E)$ is determined by these data.
  Indeed, $\tp(\bar f\rest_E/E)$ is already known to be determined.
  Let $h = \sum_i |f_i|$.
  Then $\tp(\bar f\rest_E/E)$ determines $\tp(h\rest_E/E)$, and in particular $\|h\rest_E\|$, while $\|h\|$ is determined by $\tp(\bar f)$, so $\|h\rest_{E^\perp}\|$ is determined as well.
  Alongside the facts that $h\rest_{E^\perp}$ is positive and orthogonal to $E$, this is enough to determine $\tp(h\rest_{E^\perp}/E)$, so $\tp(h/E)$ is determined.
  We may choose any realisation of this type, in a sufficiently saturated extension of $E$, and without changing $\tp(\bar f/E)$, we may assume this realisation is indeed $h$.

  We may further assume that $h\rest_E = \mathbf{1}_A$ and $h\rest_{E^\perp} = \mathbf{1}_B$ in some concrete presentation of the ambient model.
  Now $\tp(\bar f)$ determines $\tp(\bar f/h)$, which, again by \cite[Proposition~3.7]{BenYaacov-Berenstein-Henson:LpBanachLattices}, can be identified with the joint conditional distribution of $\bar f$ with respect to $\{\emptyset,A \cup B\}$ (which is essentially the same thing as the distribution of $\bar f$ restricted to $A \cup B$, with the caveat that $A \cup B$ has finite measure which is not necessarily one, i.e., is not necessarily a probability space).
  Similarly, $\tp(\bar f\rest_E/E)$ determines $\tp(\bar f\rest_E)$ and thus $\tp(\bar f\rest_E/h\rest_E)$, which can be identified with the joint conditional distribution of $\bar f\rest_E$ with respect to $\{\emptyset,A\}$.
  Subtracting, we obtain the joint conditional distribution of $\bar f\rest_{E^\perp}$ with respect to $\{\emptyset,B\}$, namely $\tp(\bar f\rest_{E^\perp}/h\rest_{E^\perp})$, and thus $\tp(\bar f\rest_{E^\perp})$ and finally $\tp(\bar f\rest_{E^\perp}/E,\bar f\rest_E)$.
  Thus $\tp(\bar f/E)$ is known and the proof of the first assertion is complete.
  The second assertion follows.
\end{proof}

\begin{cor}
  For every $p \in [1,\infty)$ and every dense subset $D \subseteq (0,1)$ the tuples
  \begin{gather*}
    \bigl( \tp(\bar f), \bE_{[t,s]}[\bar k \cdot \bar f|E] \bigr)_{t,s \in D, t<s, \bar k \in \bZ^n}
  \end{gather*}
  is a uniform canonical base for $\tp(\bar f/E)$, and
  \begin{gather*}
    \bigl( \bE_{[t,s]}[ \bar k \cdot \bar f|E] \bigr)_{t,s \in D, t<s, \bar k \in \bZ^n}
  \end{gather*}
  is a weakly uniform canonical base in the home sort.
  When $p > 1$ we may replace $\bE_{[t,s]}$ with $\bE_t$.
\end{cor}

\begin{rmk}
  At least for $p = 1$ this cannot be improved, in the sense that the types of every $\bar k \cdot \bar f$ need not determine $\tp(\bar f)$.
  Indeed, let $f_1$, $f_2$ and $f_3$ be disjoint positive functions of norm one, and let
  \begin{gather*}
    g = f_1 - f_2,
    \qquad
    h = f_1 + f_2 - 2f_3.
  \end{gather*}
  Then $\tp(kg + \ell h) = \tp(kg-\ell h)$ for all $k,\ell$, but $\tp(g,h) \neq \tp(g,-h)$.
\end{rmk}

\section{On uniform canonical bases and beautiful pairs}
\label{sec:BeautifulPairs}

It is implicitly shown by Poizat \cite{Poizat:Paires}, based on Shelah's f.c.p.\ Theorem \cite{Shelah:ClassificationTheory}, that a stable classical theory does not have the finite cover property if and only if the set of uniform canonical bases (for all types in any one given sort) is definable, rather than merely type-definable (here, a definable set in an infinite sort means a set which is closed under coordinate-wise convergence, and such that the projection to each finite sub-sort is definable in the ordinary sense).
A similar result should hold for continuous logic, where the finite cover property (and in particular Shelah's f.c.p.\ Theorem) have not yet been properly studied.
Here we concentrate on the relation between the existence of a good first order theory for beautiful pairs and the definability of the sets of uniform canonical bases.

We fix a stable theory $T$ in a language $\cL$ admitting quantifier elimination as well as a uniform canonical base map $\Cb$.
We may write the latter as $(\Cb_n)_n$, since it consists of a map for the sort of $n$-tuples for each $n$ (we shall assume that $\cL$ is single sorted, otherwise even more complex notation is required).
We define $\cL_P = \cL \cup \{P\}$, where $P$ is a new unary predicate symbol ($1$-Lipschitz, in the continuous setting).
We also define $\cL_{\Cb}$ to consist of $\cL$ along with, for each $n$, $n$-ary function symbols to the target sorts of $\Cb_n$.
We denote the (possibly infinite) tuple of these new function symbols $f_\Cb(\bar x)$, where $n = |\bar x|$.
In the continuous setting, uniform continuity moduli for the $f_\Cb$ are as per \fref{lem:UniformlyContinuousUniformCB}.
We let $\cL_{P,\Cb} = \cL_P \cup \cL_\Cb$.

By a \emph{pair} of models of $T$ we mean any elementary extension $\cN \preceq \cM \vDash T$.
We shall identify such a pair with the structures $(\cM,P)$, $(\cM,f_\Cb)$ or $(\cM,P,f_\Cb)$, as will be convenient, where $P(x) = d(x,N)$ and $f_\Cb(\bar x) = \Cb(\bar x/N)$.
The property that the predicate $P$ defines an elementary sub-structure is elementary, so the class of all pairs of models of $T$ is elementary as well, and we shall denote its theory by $T_{P,0}$.
Similarly, $T_{P,\Cb,0}$ will be the $\cL_{P,\Cb}$-theory of pairs, which consists in addition of the axioms saying that $f_\Cb(\bar x) = \Cb(\bar x/P)$.

It is easy to check that all these axioms are indeed expressible by an inductive $\cL_P$-theory and an inductive $\cL_{P,\Cb}$-theory, respectively.
Clearly, $T_{P,\Cb,0}$ is a definitional expansion of $T_{P,0}$, so we may unambiguously refer to a model of $T_{P,\Cb,0}$ as $(\cM,P)$.
On the other hand, the predicate $P$ is also superfluous in $T_{P,\Cb,0}$, since it can be recovered in the classical and continuous cases, respectively, as
\begin{gather*}
  P(x) = \exists y\, d[x=y]\bigl( y, \Cb(x/P) \bigr),
  \qquad
  P(x) = \inf_y\, d[d(x,y)]\bigl( y, \Cb(x/P) \bigr).
\end{gather*}
Since $T$ admits quantifier elimination, the formulae on the right hand side can be taken to be quantifier-free.
We may therefore express the same properties as above in an inductive $\cL_\Cb$-theory $T_{\Cb,0}$, for which $T_{P,\Cb,0}$ is merely a quantifier-free definitional expansion.
We may therefore work quite interchangeably in one setting or the other, i.e., with or without the predicate $P$.

\begin{lem}
  \label{lem:TCbAmalgamation}
  The theory $T_{\Cb,0}$ admits amalgamation over arbitrary sets (which generate isomorphic substructures).
  If $T$ is complete, then $T_{\Cb,0}$ also admits the joint embedding property.
\end{lem}
\begin{proof}
  Let $(\cM_i,P) \vDash T_{\Cb,0}$ for $i = 0,1$ and let $A$ be a common sub-structure.
  Since, for every sentence $\varphi$, we have $\varphi^{\cM_i} = d\varphi\bigl( \emptyset, \Cb(\emptyset/P) \bigr)$ and $\Cb(\emptyset/P) \in A$, we have $\cM_0 \equiv \cM_1$ even if $T$ is incomplete.
  Let $\cN$ be a sufficiently saturated model of the common complete theory.

  We first choose an embedding $\theta\colon A \rightarrow N$.
  Then, we may choose two extensions, $\theta_i\colon P(M_i) \cup A \rightarrow N$, such that in addition $\theta_0 P(M_0) \ind_{\theta A} \theta_1 P(M_1)$.
  Since $A$ is a sub-structure in $\cL_\Cb$, we know that $A \ind_{P(A)} P(M_i)$, and the same holds for the images under $\theta_i$.
  It follows that $\theta A \ind_{\theta P(A)} \theta_0 P(M_0), \theta_1 P(M_1)$, and we may choose a sub-model $P(\cN) \preceq \cN$ such that $P(N) \supseteq \theta_i P(M_i)$ and $\theta A \ind_{\theta P(A)} P(N)$.
  We may now extend each $\theta_i$ to all of $M_i$, such that $\theta_i M_i \ind_{\theta A,\theta_i P(M_i)} P(N)$.
  Then in particular $\theta_i M_i \ind_{\theta_i P(M_i)} P(N)$.
  Then $\theta_i\colon (\cM_i,P) \rightarrow (\cN,P)$ respects $\cL_\Cb$, and we are done.

  If $T$ is complete then we can amalgamate models over the (unique) $\emptyset$-generated substructure.
\end{proof}
(This argument already appears, in essence, in \cite{BenYaacov-Pillay-Vassiliev:LovelyPairs}, the only novelty is that we use the language $\cL_\Cb$ to ensure that every sub-structure is $P$-independent, i.e., verifies $A \ind_{P(A)} P$.)

It follows that a model companion of $T_{\Cb,0}$ (or of $T_{P,\Cb,0}$, this is the same thing), if it exists, eliminates quantifiers, i.e., it is a model completion.
Even if it does not exist we may still consider it as a Robinson theory in the sense of Hrushovski \cite{Hrushovski:SimplicityAndLascarGroup}.

\begin{lem}
  Modulo $T_{\Cb,0}$, the restriction of every quantifier-free $\cL_\Cb$-formula to $P$ is $\cL$-definable there.
\end{lem}
\begin{proof}
  This follows immediately from the facts that, first, for $\bar a \in P$ we have $\Cb(\bar a/P) = \Cb(\bar a/\bar a)$, and second, the map $\bar a \mapsto \Cb(\bar a/\bar a)$ is definable in $\cL$.
\end{proof}

\begin{dfn}
  Following Poizat \cite{Poizat:Paires}, we say that a pair $(\cM,P)$ of models of $T$ is \emph{beautiful} if $P$ is approximately $\aleph_0$-saturated and $\cN$ is approximately $\aleph_0$-saturated over $\cM$ (namely, the $\cL(M)$-structure $(\cN,a)_{a \in M}$ is approximately $\aleph_0$-saturated).
  We define $T_\Cb^b$ to be the $\cL_\Cb$-theory of all beautiful pairs of models of $T$.
\end{dfn}

\begin{thm}
  \label{thm:BeautifulPairs}
  Let $T$ be a stable theory with quantifier elimination and a uniform canonical base map, as above.
  Then the following are equivalent.
  \begin{enumerate}
  \item
    The image set $\img \Cb_n$ is definable for each $n$ (i.e., its projection to every finite sort is a definable set in that sort).
  \item
    If $(\cM,P) \vDash T_\Cb^b$ is $\kappa$-saturated for some $\kappa > |T|$ then $\cM$ is $\kappa$-saturated over $P(M)$.
    In particular, every sufficiently saturated model of $T_\Cb^b$ is itself a beautiful pair.
  \item
    \label{item:BeautifulPairsSaturation}
    The theory $T_{\Cb,0}$ admits a companion $T_\Cb$ such that for some $\kappa > |T|$, if $(\cM,P) \vDash T_\Cb$ is $\kappa$-saturated then $\cM$ is $\kappa$-saturated over $P(M)$.
    Moreover, such a companion is necessarily the model companion of $T_{\Cb,0}$.
  \item
    The theory $T_{\Cb,0}$ admits a model completion $T_\Cb$ (i.e., its model companion exists and eliminates quantifiers).
    \setcounter{dummycnt}{\value{enumi}}
  \end{enumerate}
    If $T$ is complete, this is further equivalent to:
  \begin{enumerate}
    \setcounter{enumi}{\value{dummycnt}}
  \item
    Let $(\cM,P) \vDash T_{\Cb,0}$, where $P(\cM)$ is $|T|^+$-saturated and $\cM$ is $|T|^+$-saturated over $P(M)$.
    Then $(\cM,P)$ is $\aleph_0$-saturated (and is, moreover, a model of the model companion).
  \item
    There exists an approximately $\aleph_0$-saturated beautiful pair $(\cM,P)$.
  \end{enumerate}
\end{thm}
\begin{proof}
  \begin{cycprf}
  \item[\impnext]
    Let $T_\Cb$ consist of $T_{\Cb,0}$ along with the axioms saying that for every canonical base $Z \in \img \Cb$ and every tuple $\bar w$, the type defined by $Z$ on the domain $P \cup \bar w$ is finitely realised (or, in the continuous setting, approximately finitely realised).
    Since the set $\img \Cb$ is definable, this axiom can indeed be expressed, and clearly $T_\Cb \subseteq T_\Cb^b$.

    Now let $(\cM,P) \vDash T_\Cb$ be $\kappa$-saturated, in which case $P(\cM)$ is $\kappa$-saturated as well.
    Let also $\bar a$ be a tuple in some elementary extension of $\cM$, $C = \Cb(\bar a/M) \in M$, and let $A \subseteq M$, $|A| < \kappa$.
    Let $\pi(\bar x)$ be the partial $\cL_\Cb$-type saying that $\bar x$ realises $\tp(\bar a/A \cup P)$, i.e., that
    \begin{gather*}
      \sup_{\bar y \in P} \bigl| \varphi(\bar x,\bar y\bar b) - d\varphi(\bar y\bar b,C) \bigr| = 0
    \end{gather*}
    for each formula $\varphi(\bar x,\bar y\bar z)$ and $\bar b \in A$.
    By $T_\Cb$, this partial type is approximately finitely realised in $(\cM,P)$, and since $|A \cup C| < \kappa$, it is realised there.
  \item[\impnext]
    It is easy to check that $T_{\Cb,0}$ and $T_\Cb^b$ are companions, whence the main assertion.
    For the moreover part, let $T_\Cb$ be any companion with the stated property, and we shall show that all its models are existentially closed.
    Indeed, let us consider an extension $(\cM,P) \subseteq (\cN,P)$ of models of $T_\Cb$, and we need to show that $(\cM,P)$ is existentially closed in $(\cN,P)$.
    Since the latter is an elementary property of an extension, we may replace the extension with any elementary extension thereof (technically speaking, we represent the extension by $(\cN,P,Q)$, where $Q(x) = d(x,M)$, and take an elementary extension of this).
    We may therefore assume that $(\cM,P)$ is $\kappa$-saturated, so in particular $P(\cM)$ is $\kappa$-saturated, and by assumption $\cM$ is $\kappa$-saturated over $P(M)$.

    Now let $\bar b \subseteq M^m$, $\bar a \in N^n$, and let $p(\bar x,\bar y) = \tp\bigl( \bar a,\bar b/ P(N) \bigr)$, $q(\bar y) = \tp\bigl( \bar b/ P(N) \bigr)$, $C = \Cb(p)$, $D = \Cb(q)$, observing that then $D \subseteq P(M)$.
    By saturation of $P(\cM)$ we may find $C' \subseteq P(M)$ such that $C' \equiv_D C$, and define $p'(\bar x,\bar y)$ to be the type over $P(N)$ defined by $C'$.
    Now, $C \equiv_D C'$ along with $q \subseteq p$ yields $q \subseteq p'$, so $p'(\bar x,\bar b)$ is consistent.
    Therefore, its restriction to $P(M),\bar b$ is realised in $M$, say by $\bar a'$.
    Then $\bar a$ and $\bar a'$ have the same quantifier-free $\cL_\Cb$-type over $\bar b$, which concludes the proof.
  \item[\impnext]
    By \fref{lem:TCbAmalgamation}, the model companion eliminates quantifiers.
    \setcounter{dummycnt}{\value{cycprfcnt}}
  \item[\impfirst]
    For $(\cM,P) \vDash T_\Cb$, let $(\img \Cb_n)^P$ denote the set of all $\Cb(p)$ as $p \in \tS_n\bigl( P(M) \bigr)$ (namely, the type-definable set $\img \Cb_n$ as interpreted in the structure $P(M)$).
    In a sufficiently saturated model of $T_\Cb$, $(\img \Cb_n)^P$ is exactly $f_\Cb(M^n)$, which is definable as a definable image of a definable set.
    It follows that it is uniformly definable in models of $T_\Cb$ (and the image of $f_\Cb$ is always dense there).
    By quantifier elimination, it is quantifier-free definable.
    It follows that $(\img \Cb_n)^P$ is $\cL$-definable in $P$, which means exactly that $\img \Cb_n$ is definable.
    \setcounter{cycprfcnt}{\value{dummycnt}}
  \item[\impnext]
    A close inspection of the argument for (ii) $\Longrightarrow$ (iii) reveals that it proves the following intermediary result: if $(\cM,P) \vDash T_{\Cb,0}$, $P(\cM)$ is $\kappa$-saturated, and $\cM$ is $\kappa$-saturated over $P(M)$, then $(\cM,P)$ is $\aleph_0$-saturated.
    In addition, such a pair is clearly a model of $T_\Cb$ as given there.
  \item[\impnext]
    Clear.
  \item[\impfirst]
    Here we also assume that $T$ is complete.
    For a formula $\varphi(\bar x,\bar y)$, with $|\bar x| = n$, let $f_\varphi$ be the definable function sending $Z$ in the target sort of $\Cb_n$ to the canonical parameter of $d\varphi(\bar y,Z)$.
    The target sort of $f_\varphi$ is equipped with a canonical definable distance $d(z,w) = \sup_{\bar y} |d\varphi(\bar y,z) - d\varphi(\bar y,w)|$.
    We may define $\Cb_\varphi = f_\varphi \circ \Cb_n$, and $\Cb_\varphi(p)$ is a (uniform) $\varphi$-canonical base for $p$.
    Through the coding of several formulae in one, we see that $\img \Cb_n$ is definable if and only if $\img \Cb_\varphi$ is for every such $\varphi$.
    We may therefore assume that $\img \Cb_\varphi$ is not definable for some $\varphi$, keeping in mind that it is always type-definable -- we may in fact identify it with a closed set in $\tS_z(T)$.

    Non definability then implies that for some $\varepsilon > 0$, the set of types $q(z)$ such that $d(q,\img \Cb_\varphi) \geq \varepsilon$ (namely, such that $q(z) \cup r(z) \vdash d(z,w) \geq \varepsilon$ for every $r \in \img \Cb_\varphi$) has an accumulation point $r \in \img \Cb_\varphi \subseteq \tS_z(T)$.
    The partial type asserting that $z \in P$, $r(z)$, and $\inf_{\bar x} g\bigl( z,f_\varphi \circ f_\Cb(\bar x) \bigr) \geq \varepsilon$ is then approximately finitely realised in $(\cM,P)$, and therefore, by the saturation assumption, realised there, say by $c$.
    Then $d\varphi(\bar y,c)$ defines a complete $\varphi$-type $p$ over $P$, which is realised in $\cM$, by the saturation assumption for $\cM$ over $P$, so $d\bigl(c, f_\varphi \circ f_\Cb(M^n) \bigr) = 0$, a contradiction.
  \end{cycprf}
\end{proof}

\begin{cor}
  The image sets $\img \Cb_n$ are definable if and only if $T_{\Cb,0}$ admits a model companion, if and only if this model companion is $T_\Cb^b$ (and is in fact a model completion).
\end{cor}

In case these equivalent conditions hold we shall simply denote the model companion by $T_\Cb$, or, in the language $\cL_P$, by $T_P$ (although it is not usually a model companion in the language $\cL_P$).
In this case we say that the class of beautiful pairs of models of $T$ is \emph{weakly elementary}, in the sense that any sufficiently saturated model of the theory of this class, $T_P$, also belongs to it.
By results of Poizat \cite{Poizat:Paires}, for a classical theory $T$ this is further equivalent to $T$ not having the finite cover property.

\begin{cor}
  If $T$ is $\aleph_0$-categorical then $T_\Cb^b$ is the model completion of $T_{\Cb,0}$.
\end{cor}
\begin{proof}
  In an $\aleph_0$-categorical theory every type-definable set is definable.
\end{proof}

Since the theories discussed in previous sections ($IHS$, $APr$ and $ALpL$) are $\aleph_0$-categorical, they satisfy the equivalent conditions of \fref{thm:BeautifulPairs} in a somewhat uninteresting fashion.
In the next section we consider a more interesting example of a non $\aleph_0$-categorical theory.

Just as we remarked in \fref{sec:UniformCB}, these results can be extended to simple theories, where beautiful pairs are replaced with lovely pairs (see \cite{BenYaacov-Pillay-Vassiliev:LovelyPairs}).
The price to be paid is either to work with hyper-imaginary sorts (which can be done relatively smoothly in continuous logic) or to assume that uniform canonical bases exist in real or imaginary sorts (which is a strong form of elimination of hyper-imaginaries).

\section{The case of $ACMVF$}
\label{sec:PairsInACMVF}

A convenient feature of condition \fref{item:BeautifulPairsSaturation} of \fref{thm:BeautifulPairs} is that it remains invariant under the addition (or removal) of imaginary sorts.
It may therefore serve as a criterion for the definability of the sets of uniform canonical bases even when these do not exist in any of the named sorts.
As an example of this, let us consider the theory $ACMVF$ of algebraically closed metric valued fields, as defined in \cite{BenYaacov:MetricValuedFields}.

\begin{thm}
  The equivalent conditions of \fref{thm:BeautifulPairs} hold for $T = ACMVF$.
\end{thm}
\begin{proof}
  We recall that a model of $ACMVF$ is not a valued field but rather a projective line $K\bP^1$ over one, equipped with a modified distance function $d(x,y) = \|x-y\| = \frac{|x-y|}{\max |x|,1 \cdot \max |y|,1}$ when $x,y \neq \infty$ and $d(x,\infty) = \frac{1}{\max |x|,1}$.

  Given a polynomial $Q(Z,\bar W,\bar V)$ over $\bZ$ and a tuple $\bar e$, let $\sqrt{Q(X,\bar e,P)}$ denote the collection of all roots of instances $Q(X,\bar e,\bar f)$ where $\bar f \in P$ (or more exactly, of the homogeneisation thereof, so $\infty$ can also be obtained a root).
  Let us define $T_\Cb$ to consist of $T_{\Cb,0}$ along with the axioms saying that every $r \in [0,1]$, $b \in P$ and $Q(X,\bar e,\bar V)$ as above there is $a$ such that $d(a,b) = d\bigl(a, \sqrt{Q(X,\bar e,P)} \bigr) = r$, or at least approximately so.
  One checks that $\sqrt{Q(X,\bar e,P)}$ is a definable set, so this is expressible in continuous logic.

  In order to check that $T_\Cb$ is indeed a companion, let us consider an instance of the axioms.
  Since the map $x \mapsto x^{-1}$ is an isometric bijection of $K\bP^1$, and for every $Q$ there is $Q'$ such that $\bigl\{ x^{-1}\colon x \in \sqrt{Q(X,\bar e,P)} \bigr\} = \sqrt{Q'(X,\bar e,P)}$, we may assume that $|b| \leq 1$.
  We may then add $a$ such that $|a-c| = \max r,|b-c|$ for all $c \in \sqrt{Q(X,\bar e,P)} \setminus \{\infty\}$, so in particular $|a| = \max |b|,r \leq 1$.
  It follows that $d\bigl( a, \sqrt{Q(X,\bar e,P)} \bigr) = d(a,b) = r$.

  Now, let $(K\bP^1,P)$ be an $\aleph_1$-saturated model of $T_\Cb$, and we claim that the opposite holds, namely, that for all $b \in P \setminus \{\infty\}$, countable set $\infty \notin A$ and $r \in (0,\infty)$, if $B$ denotes the algebraic closure of $P \cup A$ then there exists $a$ such that $|x-c| = \max r,|b-c|$ for all $c \in B$.
  We notice that if $|b| < r$ then replacing $b$ with some $b' \in P$ such that $|b'| = r$, which exists by saturation, does not change the conditions of the problem.
  We may therefore assume that $|b| \geq r$.
  If $|b|,r \leq 1$ then (by saturation of $(K\bP^1,P)$) there exists $x$ such that $d(a,B) = d(a,b) = r$, and as in the previous paragraph this $x$ is as desired.
  Otherwise, $|b| > 1$.
  Again by saturation, we find $a$ such that for all $c \in B \setminus \{\infty\}$, $|a-c| = \max r',|z-b^{-1}|$ with $r' = \frac{r}{|b|^2}$ (since $|b^{-1}|,r' \leq 1$).
  We observe that $|a| = \max |b^{-1}|, r|b^{-2}| = |b^{-1}|$, i.e., $|ab| = 1$.
  A direct calculation yields that for $c \in B \setminus \{\infty\}$, $|a^{-1} - c| = \max r\left| \frac{c}{b} \right|, |b-c|$ (the case $c = 0$ may have to be considered separately).
  We now consider three cases (keeping in mind that $|b| \geq r$):
  \begin{gather*}
    \max r\left| \frac{c}{b} \right|, |b-c| =
    \begin{cases}
    \max r\left| \frac{c}{b} \right|, |b| = |b| = \max r,|c-b|, & |c| < |b|,
    \\
    \max r,|c-b|, & |c| = |b|,
    \\
    \max r\left| \frac{c}{b} \right|, |c| = |c| = \max r,|c-b|, & |c| > |b|.
    \end{cases}
  \end{gather*}
  Our claim is thus proved.

  Now let $p(x) \in \tS_1(PA)$, and let $q$ be any global extension of $p$ to $K\bP^1$.
  Since $K\bP^1$ is $\aleph_1$-saturated, there exists $b \in K\bP^1$ such that $r = d(x,b)^{q(x)} = d\bigl[ d(x,y) \bigr]\bigl( b,\Cb(q) \bigr)$ is minimal.
  Replacing $p$ with $q\rest_{PAb}$, we may assume that $b \in A$.
  By our previous claim, there exists $a \in K\bP^1$ such that $d(a,b) = d(a,B) = r$, so this $x$ must realise $p$.

  We have thus shown that for every $\aleph_1$-saturated model $(K\bP^1,P)$ of the companion $T_\Cb$, $K\bP^1$ is $\aleph_1$-saturated over $P$, so by \fref{thm:BeautifulPairs} $T_\Cb$ is the model completion of $T_{\Cb,0}$, and every uniform canonical base map has a definable image.
\end{proof}

On the other hand, there are no canonical bases, so in particular no uniform ones, in the home sort of $ACMVF$.
Indeed, we observed in \cite{BenYaacov:MetricValuedFields} that $1$-types over models are parametrised by spheres, and it is not difficult to see that if $S$ is a sphere of non zero radius then in the home sort $\dcl(S) = \dcl(\emptyset)$.
In the case of $1$-types it is relatively easy to construct an imaginary sort in which uniform canonical bases exist.
Indeed, let us consider the set of all pairs $(a,r)$ with $a \in K\bP^1$ and $r \in [0,1]$, more conveniently written as $a_r$, and let $[a_r]$ denote the closed ball of radius $r$ around $a$.
We define
\begin{gather*}
  d(a_r, b_s) = |r-s| \vee d(a,b) \dotminus (r \wedge s).
\end{gather*}
Let us show that
\begin{gather*}
  d(a_r,c_t) \leq d(a_r,b_s) + d(b_s,c_t).
\end{gather*}
If $d(a_r,c_t) = |r-t|$ then the inequality is clear.
Otherwise, may assume that $d(a,c) \leq d(a,b)$, so
\begin{gather*}
  d(a_r,c_t) = d(a,c) \dotminus (r \wedge t)
  \leq d(a,b) \dotminus (r \wedge s) + |s-t|
  \leq d(a_r,b_s) + d(b_s,c_t).
\end{gather*}
(A somewhat less direct way of observing the same would consist of defining $\varphi(x,a,r) = d(x,a) \dotminus r$ (namely $d(x,[a_r])$), and observe that $d(a_r,b_s) = \sup_x | \varphi(x,a,r) - \varphi(x,b,s) |$.
Clearly $d(a_r,b_s) = 0$ if and only if $[a_r] = [b_s]$.)
The set of all such quotients is incomplete, and the set of completions consists, in addition to all closed balls, of all spheres over the structure (this construction can be carried out in any bounded ultra-metric structure).
In this imaginary sort, $1$-types admit a uniform canonical base map.
The case of canonical bases for $n$-types and general elimination of imaginaries for $ACMVF$ are far more complex, compare with \cite{Haskell-Hrushovski-Macpherson:EliminationOfImaginariesInValuedFields} as well as with more recent results of Hrushovski and Loeser with respect to uniform canonical bases of generically stable types in $ACVF$.

\bibliographystyle{begnac}
\bibliography{begnac}

\end{document}